\documentclass{article}%
\usepackage{amsmath}
\usepackage{amssymb}
\usepackage{amsfonts}%
\usepackage{graphicx}
\newtheorem{theorem}{Theorem}

\newtheorem{corollary}[theorem]{Corollary}

\newtheorem{lemma}[theorem]{Lemma}

\newtheorem{proposition}[theorem]{Proposition}

\newenvironment{proof}[1][Proof]{\noindent\textbf{#1.} }{\ \rule{0.5em}{0.5em}}
\begin{document}

\title{Some forms of exceptional Lie algebras}
\author{John R. Faulkner\\University of Virginia\\Department of Mathematics\\Kerchof Hall, P. O. Box 400137\\Charlottesville VA 22904-4137\\USA\smallskip\\email: jrf@virginia.edu}
\maketitle

\begin{abstract}
Some forms of Lie algebras of types $E_{6}$, $E_{7}$, and $E_{8}$ are
constructed using the exterior cube of a rank $9$ finitely generated
projective module.

\end{abstract}

\section{Introduction}

Let $\mathcal{G}(\mathbb{C)}$ be a simple Lie algebra over $\mathbb{C}$ of
type $X_{l}$ and let $\mathcal{G}(\mathbb{Z})$ be the $\mathbb{Z}$-span of a
Chevalley basis of $\mathcal{G}(\mathbb{C)}$. \ We say that a Lie algebra
$\mathcal{G}$ over a unitary commutative ring $k$ is a \textit{form} of $X_{l}
$ if there is a faithfully flat, commutative, unital $k$-algebra $F$ with
$\mathcal{G}_{F}\cong\mathcal{G}(\mathbb{Z})_{F}$ where $\mathcal{G}%
_{F}=\mathcal{G}\otimes_{k}F$ as a $F$-module. \ The main purpose of this
paper is the construction of some forms of $E_{6}$, $E_{7}$, and $E_{8}$ using
the exterior cube of a rank $9$ finitely generated projective module. \ In
\S 2, we develop the necessary exterior algebra and localization machinery.
\ In \S 3, we construct a Lie algebra from the exterior cube of a rank $9$
finitely generated projective module, and then give a twisted version of the
construction. \ In \S 4, we show that the Lie algebras are forms of $E_{8}$
and identify some subalgebras which are forms of $E_{6}$ and $E_{7}$.

\bigskip

\section{Preliminary results}

Let $k$ be a unitary commutative ring. \ Throughout, we require that a
$k$-module $M$ be unital; i.e., $1x=x$ for $x\in M$. \ Let $M^{\ast}%
=Hom_{k}(M,k)$, the dual module. \ Recall that a $k$-module $M$ is
\textit{projective} if $M$ is a direct summand of a free module (\cite{B88}%
,II.2.2). \ Moreover, $M$ is a finitely generated projective module if and
only if $M$ is a direct summand of a free module of finite rank (\cite{B88}%
,II.2.2). \ Let $M$ and $N$ be finitely generated projective modules. \ Then
$M^{\ast}$ and $M\otimes N$ are also finitely generated projective
((\cite{B88},II.2.6,II.3.7), and we may identify $M$ with $M^{\ast\ast}$ where
$m(\phi)=\phi(m)$ for $m\in M$ and $\phi\in M^{\ast}$ (\cite{B88},II.2.7).
\ Moreover, the linear map%
\[
M\otimes M^{\ast}\rightarrow End(M)
\]
with $m\otimes\phi\rightarrow m\phi$ where $(m\phi)(m^{\prime})=\phi
(m^{\prime})m$ is bijective (\cite{B88},II.4.2). \ Thus, we can define the
\textit{trace} function $tr$ on $End(M)$ as the unique linear map with
$tr(m\phi)=\phi(m)$. \ Since%
\[
tr((m\phi)(m^{\prime}\phi^{\prime}))=\phi^{\prime}(m)\phi(m^{\prime}),
\]
we see that $tr(\alpha\beta)=tr(\beta\alpha)$ for $\alpha,\beta\in End(M)$.
\ Letting $gl(M)=End(M)$ with Lie product $[\alpha,\beta]=\alpha\beta
-\beta\alpha$, we see%
\[
\lbrack gl(M),gl(M)]\subset sl(M):=\{\alpha\in gl(M):tr(\alpha)=0\},
\]
so $sl(M)$ is an ideal in $gl(M)$.

Let $k$-alg denote the category of commutative unital $k$-algebras. \ If
$K\in$ $k$-alg and $M$, $N$ are $k$-modules, let $M_{K}=M\otimes_{k}K$ as a
$K$-module. \ If $M$ is a finitely generated projective $k$-module, then
\begin{align*}
(M\otimes_{k}N)_{K}  & \cong M_{K}\otimes_{K}N_{K},\\
(M^{\ast})_{K}  & \cong(M_{K})^{\ast},\\
gl(M)_{K}  & \cong gl(M_{K})
\end{align*}
via canonical isomorphisms (\cite{B88},II.5.1,II.5.4).

If $\mathfrak{p}$ is a prime ideal of $k$, let $k_{\mathfrak{p}}%
=(k\diagdown\mathfrak{p)}^{-1}k$ be the \textit{localization} of $k$ at
$\mathfrak{p}$ and $M_{\mathfrak{p}}=M_{k_{\mathfrak{p}}}$ be the
\textit{localization} of $M$ at $\mathfrak{p}$ (\cite{B89},II). \ If $M$ is
finitely generated projective, then $M_{\mathfrak{p}}$ is a free
$k_{\mathfrak{p}}$-module of finite rank (\cite{B89},II.5.2). \ If
$M_{\mathfrak{p}}$ has rank $n$ for all prime ideals $\mathfrak{p}$ of $k$, we
say $M$ has \textit{rank} $n$. \ In this case, $M_{K}$ has rank $n$ for all
$K\in k$-alg (\cite{B89},II.5.3). \ Moreover, if $M,N$ are finitely generated
projective modules and $\alpha\in Hom(M,N)$, then $\alpha$ is injective
(respectively, surjective, bijective, zero) if and only if $\alpha
_{\mathfrak{p}}=\alpha\otimes Id_{k_{\mathfrak{p}}}\in Hom(M_{\mathfrak{p}%
},N_{\mathfrak{p}})$ is injective (respectively, surjective, bijective, zero)
for each prime ideal $\mathfrak{p}$ (\cite{B89}, II.3.3). \ This allows the
transfer of multilinear identities using localization as follows: if
$M_{1},\ldots,M_{l},N$ are finite generated projective modules and
\[
\mu:M_{1}\times\cdots\times M_{l}\rightarrow N
\]
is a $k$-multilinear map, then for $K\in$ $k$-alg there is a unique
$K$-multilinear map%

\[
\mu_{K}:M_{1K}\times\cdots\times M_{lK}\rightarrow N_{K}%
\]
with
\begin{equation}
\mu_{K}(m_{1}\otimes1,\ldots,m_{l}\otimes1)=\mu(m_{1},\ldots,m_{l}%
)\otimes1.\label{eq: multilinear extension}%
\end{equation}
We claim $\mu_{\mathfrak{p}}=0$ for each prime ideal $\mathfrak{p}$ implies
$\mu=0$. \ Indeed, $M_{1}\otimes\cdots\otimes M_{l}$ is finitely generated
projective and $\mu$ induces a linear map
\[
\tilde{\mu}:M_{1}\otimes\cdots\otimes M_{l}\rightarrow N
\]
with each $(\tilde{\mu})_{\mathfrak{p}}=\widetilde{(\mu_{\mathfrak{p}})}=0$,
so $\tilde{\mu}=0$ and $\mu=0$.

Recall $F\in k$-alg is \textit{faithfully flat} provided a sequence
$M^{\prime}\rightarrow M\rightarrow M^{\prime\prime}$ is exact if and only if
the induced sequence $M_{F}^{\prime}\rightarrow M_{F}\rightarrow M_{F}%
^{\prime\prime}$ is exact. \ We shall need the following example of a
faithfully flat algebra. \ Recall a quadratic form $q$ on $M$ is
\textit{nonsingular} if $a\rightarrow q(a,~)$ is an isomorphism $M\rightarrow
M^{\ast}$ where
\[
q(a,b):=q(a+b)-q(a)-q(b).\
\]
We say that $K\in k$-alg is a \textit{quadratic \'{e}tale algebra} if $K$ is a
finitely generated projective $k$-module of rank $2$ with a nonsingular
quadratic form $n$ admitting composition; i.e.,%
\[
n(ab)=n(a)n(b).
\]
\ We did not find a suitable reference for the following result, so we include
a proof communicated to us by H. Petersson.

\begin{proposition}
\label{prop: quadratic etale}If $K$ is a quadratic \'{e}tale algebra over $k$,
then $K$ is faithfully flat and $K_{K}\cong K\oplus K$.
\end{proposition}

\begin{proof}
For each maximal ideal $m$ of $k$, $K_{m}$ is a nonzero free $k_{m}$-module,
and hence faithfully flat (\cite{B89}, II.3.1). \ Thus, $K$ is faithfully flat
over $k$ (\cite{B89}, II.3.4). \ Let $t(a)=n(a,1)$ and $\bar{a}=t(a)1-a$, for
$a\in K$. \ We claim $\eta:K_{K}\rightarrow K\oplus K$ with $\eta(a\otimes
b)=ab\oplus\bar{a}b$ is a $K$-algebra isomorphism. \ Using localization, it
suffices to assume that $k$ is a field. \ In this case, it is well-known that
$K$ is commutative, $n(1)=1$, $a\rightarrow\bar{a}$ is an involution, and
$a^{-1}=n(a)^{-1}\bar{a}$, if $n(a)\neq0$. Thus, $\eta$ is a homomorphism of
$K$-algebras with involution where $K\oplus K$ has the exchange involution.
\ By dimensions, it suffices to show $\eta$ is surjective. \ Let $1,u$ be a
$k$-basis of $K$. \ We see%
\begin{align*}
n(\bar{u}-u)  & =n(t(u)1-2u)\\
& =4n(u)-t(u)^{2}\\
& =\det\left[
\begin{array}
[c]{cc}%
n(1,1) & n(1,u)\\
n(u,1) & n(u,u)
\end{array}
\right]  \neq0
\end{align*}
since $n$ is nonsingular, so $\bar{u}-u$ is invertible. \ Now $\eta
(u\otimes1-1\otimes u)=0\oplus(\bar{u}-u)$, so $\eta(K_{K})$ contains
$0\oplus1$,$1\oplus0=\overline{0\oplus1}$, and hence $K\oplus K$.
\end{proof}

\bigskip

We now recall some facts about exterior algebras. \ For more details see
\cite{B88}. \ Let $M$ be a $k$-module and form the exterior algebra
$\Lambda(M)$ with the standard $\mathbb{Z}$-grading%
\[
\Lambda(M)=\sum_{i\geq0}\Lambda_{i}(M),
\]
and write $\left\vert x\right\vert =i$, if $x\in\Lambda_{i}(M)$. \ For
simplicity of notation, we write the product in $\Lambda(M)$ as $xy$ rather
than the usual $x\wedge y$. \ We have $\Lambda(M)_{K}\cong\Lambda(M_{K})$ via
a canonical isomorphism (\cite{B88},III.7.5). \ If $M$ is finitely generated
projective, then so is $\Lambda(M)$ (\cite{B88},III.7.8). \ If $\alpha\in
Hom(M,N)$, then $\alpha$ extends uniquely to a graded algebra homomorphism
$\theta_{\alpha}:\Lambda(M)\rightarrow\Lambda(N)$. \ Also, if $\alpha\in
gl(M)$, then there is a unique extension of $\alpha$ to a derivation
$D_{\alpha}$ of $\Lambda(M)$. \ Thus, $\Lambda(M)$ is a module for the Lie
algebra $gl(M)$ via $(\alpha,x)\rightarrow D_{\alpha}(x)$. \ Similarly, if
$\phi\in M^{\ast}$, then there is a unique extension of $\phi$ to an
anti-derivation (or odd super derivation) $\Delta_{\phi}$ of $\Lambda(M)$.
\ Recall $\Delta$ is an \textit{anti-derivation} if
\[
\Delta(xy)=\Delta(x)y+(-1)^{\left\vert x\right\vert }x\Delta(y)
\]
if $x$ is homogeneous. \ One can show by induction on $i$ that
\begin{equation}
\Delta_{\phi}(\Lambda_{i}(M))\subset\Lambda_{i-1}%
(M),\label{eq: del_phi degree -1}%
\end{equation}
where $\Lambda_{l}(M)=0$ for $l<0$, and $\Delta_{\phi}^{2}=0$. \ Thus, the
universal property for $\Lambda(M^{\ast})$ shows that $\phi\rightarrow
\Delta_{\phi}$ extends to a homomorphism $\Delta:\Lambda(M^{\ast})$ into
$End_{k}(\Lambda(M))$, so we can view $\Lambda(M)$ as a left module for the
associative algebra $\Lambda(M^{\ast})$ with $\xi\cdot x=\Delta_{\xi}(x)$ for
$\xi\in\Lambda(M^{\ast})$, $x\in\Lambda(M)$. \ Using
(\ref{eq: del_phi degree -1}), we see
\[
\Lambda_{i}(M^{\ast})\cdot\Lambda_{j}(M)\subset\Lambda_{j-i}(M).
\]

Let $M$ be a finitely generated projective $k$-module. \ Since $M^{\ast\ast
}=M$, we can reverse the roles of $M$ and $M^{\ast}$ and see that
$\Lambda(M^{\ast})$ is a left module for $\Lambda(M)$ via $x\cdot\xi$. \ Also,
we can identify $\Lambda_{i}(M^{\ast})$ with $\Lambda_{i}(M)^{\ast}$ where
$\xi(x)=\xi\cdot x$ for $\xi\in\Lambda_{i}(M^{\ast})$, $x\in\Lambda_{i}(M)$
(\cite{B88},III.11.5).

For $\alpha\in Hom(M,N)$, let $\alpha^{\ast}\in Hom(N^{\ast},M^{\ast})$ with
$\alpha^{\ast}(\phi)=\phi\alpha$ for $\phi\in N^{\ast}$. \ Thus,
$\alpha\rightarrow-\alpha^{\ast}$ is a Lie algebra homomorphism
$gl(M)\rightarrow gl(M^{\ast})$ and $\Lambda(M^{\ast})$ is a module for
$gl(M)$ via $(\alpha,\xi)\rightarrow D_{-\alpha^{\ast}}(\xi)$.

\bigskip

\begin{lemma}
\label{lem: dot action}Let $l\leq n$ and let $S\subset S_{n}$ be such that
$\sigma\rightarrow\sigma\mid_{\{1,\ldots,l\}}$ is a bijection of $S$ with the
set of all injections
\[
\{1,\ldots,l\}\rightarrow\{1,\ldots,n\}.
\]
For $\phi_{i}\in M^{\ast},m_{j}\in M$, we have%
\[
(\phi_{l}\phi_{l-1}\cdots\phi_{1})\cdot(m_{1}m_{2}\cdots m_{n})=\sum
_{\sigma\in S}(-1)^{\sigma}\phi_{1}(m_{\sigma1})\cdots\phi_{l}(m_{\sigma
l})m_{\sigma(l+1)}\cdots m_{\sigma n}.
\]

\end{lemma}

\begin{proof}
Applying $\Delta_{\phi_{l}}\cdots\Delta_{\phi_{1}}$ to $m_{1}m_{2}\cdots
m_{n}$, we get terms
\[
\pm\phi_{1}(m_{i_{1}})\cdots\phi_{l}(m_{i_{l}})m_{i_{l+1}}\cdots m_{i_{n}}%
\]
with the sign depending only on $i_{1},\ldots,i_{n}$. \ There is a unique
$\sigma\in S$ with $\sigma(j)=i_{j}$ for $1\leq j\leq l$. \ After suitably
rearranging the factors of $m_{i_{l+1}}\cdots m_{i_{n}}$, we can assume
$i_{j}=\sigma(j)$ for all $j$. \ Thus,%
\[
(\phi_{l}\phi_{l-1}\cdots\phi_{1})\cdot(m_{1}m_{2}\cdots m_{n})=\sum
_{\sigma\in S}\varepsilon_{\sigma}\phi_{1}(m_{\sigma1})\cdots\phi
_{l}(m_{\sigma l})m_{\sigma(l+1)}\cdots m_{\sigma(n)}%
\]
for some $\varepsilon_{\sigma}=\pm1$, depending only on $\sigma$ \ In
particular, if $m_{1},\ldots,m_{n}$ is the basis of a vector space $V$ over a
field of characteristic not $2$ and $\phi_{i}\in V^{\ast}$ with $\phi
_{i}(m_{j})=\delta_{ij}$, then for $\tau\in S$, we have%
\begin{align*}
m_{l+1}\cdots m_{n}  & =(\phi_{l}\cdots\phi_{1})\cdot(m_{1}\cdots m_{n})\\
& =(-1)^{\tau}(\phi_{l}\cdots\phi_{1})\cdot(m_{\tau^{-1}1}\cdots m_{\tau
^{-1}n})\\
& =(-1)^{\tau}\sum_{\sigma\in S}\varepsilon_{\sigma}\phi_{1}(m_{\tau
^{-1}\sigma1})\cdots\phi_{l}(m_{\tau^{-1}\sigma l})m_{\tau^{-1}\sigma
(l+1)}\cdots m_{\tau^{-1}\sigma(n)}\\
& =(-1)^{\tau}\varepsilon_{\tau}m_{l+1}\cdots m_{n}%
\end{align*}
and $\varepsilon_{\tau}=(-1)^{\tau}$.
\end{proof}

\bigskip

We remark that if $l=1$ in Lemma \ref{lem: dot action}, we can take $S=C_{n}$,
the cyclic group generated by the permutation $(1,\ldots,n)$.

If $\alpha\in gl(M)$ and $\phi\in M^{\ast}$, then $[D_{\alpha},\Delta_{\phi}]$
is an antiderivation with%
\[
\lbrack D_{\alpha},\Delta_{\phi}](m)=D_{\alpha}(\phi(m))-\phi(\alpha
m)=\Delta_{-\alpha^{\ast}(\phi)}(m),
\]
for $m\in M$ \ \ Thus, $[D_{\alpha},\Delta_{\phi}]=\Delta_{-\alpha^{\ast}%
(\phi)}=\Delta_{D_{-\alpha^{\ast}}(\phi)}$. \ Since $\Delta$ is a
homomorphism, we have
\[
\lbrack D_{\alpha},\Delta_{\xi}]=\Delta_{D_{-\alpha^{\ast}}(\xi)}%
\]
for all $\xi\in\Lambda(M^{\ast})$, so%
\begin{equation}
D_{\alpha}(\xi\cdot x)=D_{-\alpha^{\ast}}(\xi)\cdot x+\xi\cdot D_{\alpha
}(x),\label{eq: derivation of dot}%
\end{equation}
for all $x\in\Lambda(M)$.

\bigskip

Let \ $M$ be finitely generated projective. \ For $x\in\Lambda_{l}(M),\xi
\in\Lambda_{l}(M^{\ast})$, define $e(x,\xi)\in End(M)$ by%
\[
e(x,\xi)(m)=(m\cdot\xi)\cdot x\in\Lambda_{l-1}(M^{\ast})\cdot\Lambda
_{l}(M)\subset M
\]
for $m\in M$. \ We also have $e(\xi,x)\in End(M^{\ast})$.

\begin{lemma}
\label{lem: e properties}Let $M$ be a finitely generated projective module,
and let $x,y,z\in\Lambda_{l}(M)$, $\xi\in\Lambda_{l}(M^{\ast})$, and $\mu
\in\Lambda_{3l}(M^{\ast})$. \ We have

\qquad(i) $x\cdot\xi=\xi\cdot x,$

\qquad(ii) $e(x,\xi)^{\ast}=e(\xi,x),$

\qquad(iii) if $\phi_{1},\ldots,\phi_{l}\in M^{\ast}$, then%
\[
D_{e(x,\phi_{1}\cdots\phi_{l})}=\sum_{\sigma\in C_{l}}(-1)^{\sigma}%
((\phi_{\sigma2}\cdots\phi_{\sigma l})\cdot x)\Delta_{\phi_{\sigma1}},
\]
where $C_{l}$ is the cyclic group generated by the permutation $(1,\ldots,l) $,

\qquad(iv) $tr(e(x,\xi))=l\xi\cdot x,$

\qquad(v) \ $e(xyz,\mu)=\sum\limits_{x,y,z\circlearrowleft}e(x,(yz)\cdot\mu)$,
where the sum is over all cyclic permutations of $x,y,z,$

\qquad(vi) \ if $l=3$, then $\xi\cdot(xy)=(\xi\cdot x)y-D_{e(x,\xi
)}y+D_{e(y,\xi)}x-(\xi\cdot y)x.$
\end{lemma}

\begin{proof}
Using Lemma \ref{lem: dot action}, we have%
\begin{align*}
(\phi_{l}\phi_{l-1}\cdots\phi_{1})\cdot(m_{1}m_{2}\cdots m_{l})  &
=\sum_{\sigma\in S_{l}}(-1)^{\sigma}\phi_{1}(m_{\sigma1})\cdots\phi
_{l}(m_{\sigma l})\\
& =\sum_{\sigma\in S_{l}}(-1)^{\sigma}m_{\sigma1}(\phi_{1})\cdots m_{\sigma
l}(\phi_{l})\\
& =\sum_{\sigma\in S_{l}}(-1)^{\sigma}m_{1}(\phi_{\sigma1})\cdots m_{l}%
(\phi_{\sigma l})\\
& =(m_{1}m_{2}\cdots m_{l})\cdot(\phi_{l}\phi_{l-1}\cdots\phi_{1})
\end{align*}
for $m_{i}\in M,\phi_{i}\in M^{\ast}$, showing (i). \ For $\phi\in M^{\ast
},m\in M$, we have%
\begin{align*}
(e(x,\xi)^{\ast}(\phi))(m)  & =\phi(e(x,\xi)(m))\\
& =\phi\cdot((m\cdot\xi)\cdot x)=(\phi(m\cdot\xi))\cdot x\\
& =(-1)^{l-1}((m\cdot\xi)\phi)\cdot x=(-1)^{l-1}(m\cdot\xi)\cdot(\phi\cdot
x)\\
& =(-1)^{l-1}(\phi\cdot x)\cdot(m\cdot\xi)=m\cdot((\phi\cdot x)\cdot\xi)\\
& =(e(\xi,x)(\phi))(m)
\end{align*}
showing (ii).

If $m\in M,\phi\in M^{\ast}$ it is easy to see that $m\Delta_{\phi
}:x\rightarrow m(\phi\cdot x)$ is a derivation of $\Lambda(M)$, so
$m\Delta_{\phi}=D_{m\phi}$. \ By Lemma \ref{lem: dot action}, we have%
\begin{align*}
e(x,\phi_{1}\cdots\phi_{l})m  & =(m\cdot(\phi_{1}\cdots\phi_{l}))\cdot x\\
& =\sum_{\sigma\in C_{l}}(-1)^{\sigma}((m\cdot\phi_{\sigma1})(\phi_{\sigma
2}\cdots\phi_{\sigma l}))\cdot x\\
& =\sum_{\sigma\in C_{l}}(-1)^{\sigma}((\phi_{\sigma2}\cdots\phi_{\sigma
l})\cdot x)\Delta_{\phi_{\sigma1}}(m),
\end{align*}
for $m\in M$, and (iii) follows. \ Also,
\begin{align*}
tr(e(x,\phi_{1}\cdots\phi_{l}))  & =\sum_{\sigma\in C_{l}}(-1)^{\sigma}%
\phi_{\sigma1}((\phi_{\sigma2}\cdots\phi_{\sigma l})\cdot x)\\
& =\sum_{\sigma\in C_{l}}(-1)^{\sigma}\phi_{\sigma1}\cdot((\phi_{\sigma
2}\cdots\phi_{\sigma l})\cdot x)\\
& =\sum_{\sigma\in C_{l}}(-1)^{\sigma}(\phi_{\sigma1}\phi_{\sigma2}\cdots
\phi_{\sigma l})\cdot x\\
& =l(\phi_{1}\cdots\phi_{l})\cdot x,
\end{align*}
showing (iv). \ For (v), we see%
\[
\phi\cdot(xyz)=(\phi\cdot x)yz+(-1)^{l}x(\phi\cdot y)z+xy(\phi\cdot
z)=\sum\limits_{x,y,z\circlearrowleft}(\phi\cdot x)yz,
\]
for $\phi\in M^{\ast}$, so%
\begin{align*}
e(\mu,xyz)\phi & =(\sum\limits_{x,y,z\circlearrowleft}(\phi\cdot x)yz)\cdot
\mu\\
& =\sum\limits_{x,y,z\circlearrowleft}(\phi\cdot x)\cdot((yz)\cdot\mu)\\
& =\sum\limits_{x,y,z\circlearrowleft}e((yz)\cdot\mu,x)\phi.
\end{align*}
Thus, $e(\mu,xyz)=\sum\limits_{x,y,z\circlearrowleft}e((yz)\cdot\mu,x)$, and
(v) follows from (ii). \ Finally, if $\xi=\phi_{1}\phi_{2}\phi_{3}$, then%
\begin{align*}
\xi\cdot(xy)  & =(\xi\cdot x)y-\sum_{\sigma\in C_{3}}(-1)^{\sigma}%
((\phi_{\sigma1}\phi_{\sigma2})\cdot x)(\phi_{\sigma3}\cdot y)\\
& +\sum_{\sigma\in C_{3}}(-1)^{\sigma}(\phi_{\sigma1}\cdot x)((\phi_{\sigma
2}\phi_{\sigma3})\cdot y)-x(\xi\cdot y)\\
& =(\xi\cdot x)y-D_{e(x,\xi)}y+D_{e(y,\xi)}x-(\xi\cdot y)x,
\end{align*}
showing (vi).
\end{proof}

\bigskip

\begin{lemma}
\label{lem: rank n}Let $M$ be a finitely generated projective module of rank
$n$.

(i) $(x\cdot\mu)\cdot u=(\mu\cdot u)x$, for $x\in\Lambda(M)$, $u\in\Lambda
_{n}(M)$, $\mu\in\Lambda_{n}(M^{\ast})$.

(ii) The following are equivalent:

\qquad(a) there exist $u\in\Lambda_{n}(M)$ and $\mu\in\Lambda_{n}(M^{\ast})$
with $\mu\cdot u=1$,

\qquad(b) \ $\Lambda_{n}(M)$ is free of rank $1$.

(iii) $D_{\alpha}(u)=tr(\alpha)u$ for $\alpha\in gl(M)$, $u\in\Lambda_{n}(M)$.
\end{lemma}

\begin{proof}
We first show (i) in case $M$ is a free module of rank $n$. \ Since
$\Lambda_{n}(M)$ is free of rank $1$, we may assume $x=m_{l}\cdots m_{1}$,
$u=m_{n}\cdots m_{1}$, and $\mu=\phi_{1}\cdots\phi_{n}$ where $m_{1}%
,\ldots,m_{n}$ is a basis for $M$ and $\phi_{1},\ldots,\phi_{n}$ is the dual
basis of $M^{\ast}$; i.e., $\phi_{i}(m_{j})=\delta_{ij}$. \ We have\
\begin{align*}
((m_{l}\cdots m_{1})\cdot(\phi_{1}\cdots\phi_{n}))\cdot(m_{n}\cdots m_{1})  &
=(\phi_{l+1}\cdots\phi_{n})\cdot(m_{n}\cdots m_{1})\\
& =m_{l}\cdots m_{1}\\
& =((\phi_{1}\cdots\phi_{n})\cdot(m_{n}\cdots m_{1}))m_{l}\cdots m_{1},
\end{align*}
showing (i) in this case. \ To show the general case, we observe that
$\Lambda(M)$, $\Lambda_{n}(M^{\ast})$, and $\Lambda_{n}(M)$ are finitely
generated projective, and that we can identify $\Lambda_{l}(M)_{\mathfrak{p}}$
with $\Lambda_{l}(M_{\mathfrak{p}})$. \ Since the trilinear identity (i) holds
for the free $k_{\mathfrak{p}}$-module $M_{\mathfrak{p}}$ of rank $n$ for each
$\mathfrak{p}$, it holds for $M$.

If (a) holds, then $q=(q\cdot\mu)\cdot u=(\mu\cdot q)u$ for $q\in\Lambda
_{n}(M)$ by (i). \ Thus, $q\rightarrow\mu\cdot q$ is a linear map $\Lambda
_{n}(M)\rightarrow k$ with inverse $a\rightarrow au$, and (b) holds.
\ Conversely, if $\mu:\Lambda_{n}(M)\rightarrow k$\ is an isomorphism, then
$\mu\in\Lambda_{n}(M)^{\ast}=\Lambda_{n}(M^{\ast})$ and $\mu\cdot u=\mu(u)=1$,
so (a) holds, showing (ii). \ Let%
\[
\lambda:gl(M)\otimes\Lambda_{n}(M)\rightarrow\Lambda_{n}(M)
\]
be the linear map with $\lambda(\alpha\otimes u)=D_{\alpha}(u)-tr(\alpha)u$.
\ Since (iii) holds for free modules, $\lambda_{\mathfrak{p}}=0$ for all prime
ideals $\mathfrak{p}$ of $k$, so $\lambda=0$ and (iii) holds.
\end{proof}

\bigskip

We remark that if condition (ii)(a) in Lemma \ref{lem: rank n} holds, then
$\{u\}$ is a basis for $\Lambda_{n}(M)$, $\{\mu\}$ is a basis for $\Lambda
_{n}(M^{\ast})$, and $\mu$ is uniquely determined by $u$.

\bigskip

\section{Constructions of Lie algebras}

Let $M$ be a finitely generated projective module of rank $9$ and suppose
there exist $u\in\Lambda_{9}(M)$ and $\mu\in\Lambda_{9}(M^{\ast})$ with
$\mu\cdot u=1$. \ The Lie algebra $gl(M)$ acts on $\Lambda_{3}(M)$ via
$\rho_{M}:\alpha\rightarrow D_{\alpha}\mid_{\Lambda_{3}(M)}$. \ Clearly,
$\widetilde{gl}(M):=\rho_{M}(gl(M))+kId_{\Lambda_{3}(M)}$ is a Lie algebra.
\ Since $\rho_{M}(Id_{M})=3Id_{\Lambda_{3}(M)}$, we see that $\widetilde
{gl}(M)=\rho_{M}(gl(M))$ if $\frac{1}{3}\in k$. \ Suppose $\beta\in
gl(\Lambda_{3}(M))$ extends to a derivation $d_{\beta}$ of the subalgebra
\[
\Lambda_{(3)}(M):=k\oplus\Lambda_{3}(M)\oplus\Lambda_{6}(M)\oplus\Lambda
_{9}(M)
\]
of $\Lambda(M)$. \ Since $\beta$ uniquely determines $d_{\beta}$, we can
define $T(\beta)=\mu\cdot d_{\beta}(u)$. \ If $\alpha\in gl(M)$, then
$\rho_{M}(\alpha)$ and $Id_{\Lambda_{3}(M)}$ extend to derivations of
$\Lambda_{(3)}(M)$ with $d_{\rho_{M}(\alpha)}=D_{\alpha}\mid_{\Lambda
_{(3)}(M)}$ and $d_{Id_{\Lambda_{3}(M)}}(x)=rx$ for $x\in\Lambda_{3r}(M)$.
\ Thus, each $\beta\in\widetilde{gl}(M)$ extends to a derivation $d_{\beta}$
of $\Lambda_{(3)}(M)$, and we have defined a linear map $T:\widetilde
{gl}(M)\rightarrow k$ with $T(\rho_{M}(\alpha))=tr(\alpha)$ by Lemma
\ref{lem: rank n}(iii) and $T(Id_{\Lambda_{3}(M)})=3$. \ Set $\widetilde
{sl}(M)=\{\beta\in\widetilde{gl}(M):T(\beta)=0\}$, so $\widetilde{sl}%
(M)=\rho_{M}(sl(M))$ if $\frac{1}{3}\in k$. \ Note that
\[
\lbrack\widetilde{gl}(M),\widetilde{gl}(M)]\subset\rho_{M}%
([gl(M),gl(M)])\subset\rho_{M}(sl(M))\subset\widetilde{sl}(M),
\]
so $\widetilde{sl}(M)$ is an ideal of $\widetilde{gl}(M)$. \ Note that
$\widetilde{gl}(M)$ is a Lie algebra of linear transformations of $\Lambda
_{3}(M)$ with the contragredient action on $\Lambda_{3}(M)^{\ast}=\Lambda
_{3}(M^{\ast})$. \ In particular, (\ref{eq: derivation of dot}) shows
\begin{equation}
\rho_{M}(\alpha)^{\ast}=D_{\alpha^{\ast}}\mid_{\Lambda_{3}(M^{\ast})}%
=\rho_{M^{\ast}}(\alpha^{\ast})\text{ for }\alpha\in
gl(M).\label{eq: rho star}%
\end{equation}

\bigskip

\begin{theorem}
\label{thm: split}Let $M$ be a finitely generated projective module of rank
$9$ and suppose there exist $u\in\Lambda_{9}(M)$ and $\mu\in\Lambda
_{9}(M^{\ast})$ with $\mu\cdot u=1$. \ Then
\[
\mathcal{G}(M,u)=\widetilde{sl}(M)\oplus\Lambda_{3}(M)\oplus\Lambda
_{3}(M^{\ast})
\]
is a Lie algebra with skew symmetric product given by
\begin{align*}
\lbrack\alpha,\beta]  & =\alpha\beta-\beta\alpha,\\
\lbrack\alpha,x]  & =\alpha(x),\ [\alpha,\xi]=-\alpha^{\ast}(\xi),\\
\lbrack x,y]  & =(xy)\cdot\mu,\ [\xi,\psi]=(\xi\psi)\cdot u,\\
\lbrack x,\xi]  & =\delta(x,\xi):=\rho(e(x,y))-(x\cdot\xi)Id_{\Lambda_{3}(M)}%
\end{align*}
for $\alpha,\beta\in\widetilde{sl}(M)$, $x,y\in\Lambda_{3}(M)$, and $\xi
,\psi\in\Lambda_{3}(M^{\ast})$.
\end{theorem}

\begin{proof}
We recall that Lemma \ref{lem: rank n}(ii) shows that $\mu$ is uniquely
determined by $u$. \ Also, Lemma \ref{lem: e properties}(iv) shows that
$\delta(x,y)\in\widetilde{sl}(M)$. \ It suffices to check the Jacobi identity
\[
J(z_{1},z_{2},z_{3})=[[z_{1}z_{2}]z_{3}]+[[z_{2}z_{3}]z_{1}]+[[z_{3}%
z_{1}]z_{2}]=0
\]
for $z_{i}\in\widetilde{sl}(M)\cup\Lambda_{3}(M)\cup\Lambda_{3}(M^{\ast})$.
\ Moreover, since the product is skew-symmetric,
\[
J(z_{1},z_{2},z_{3})=0\text{ implies }J(z_{\pi1},z_{\pi2},z_{\pi3})=0
\]
for any $\pi\in S_{3}$. \ Since $\widetilde{sl}(M)$ is a Lie algebra of linear
transformations of $\Lambda_{3}(M)$ with the contragredient action on
$\Lambda_{3}(M)^{\ast}=\Lambda_{3}(M^{\ast})$, the Jacobi identity holds if
two or more of $z_{i}$ are in $\widetilde{sl}(M)$. \ Interchanging the roles
of $M$ and $M^{\ast}$, if necessary, we are left with the following cases with
$\alpha\in\widetilde{sl}(M)$, $x,y,z\in\Lambda_{3}(M)$, $\xi\in\Lambda
_{3}(M^{\ast})$:

Case 1: $J(\alpha,x,\xi)$. \ We know that $gl(M)$ acts as derivations of
$\Lambda(M)$ via $\gamma\rightarrow D_{\gamma}$, and as derivations of
$\Lambda(M^{\ast})$ via $\gamma\rightarrow-D_{\gamma^{\ast}}$. \ \ Also, these
actions are derivations of the products $\Lambda(M^{\ast})\cdot\Lambda(M)$ and
$\Lambda(M)\cdot\Lambda(M^{\ast})$ by (\ref{eq: derivation of dot}). \ Thus,
$gl(M)$ acts as derivations of the triple product
\[
\delta(x,\xi)(y)=D_{e(x,\xi)}(y)-(x\cdot\xi)y.
\]
Now $End(\Lambda_{3}(M))$ acts on $\Lambda_{3}(M^{\ast})$ via $\alpha
\rightarrow-\alpha^{\ast}$. \ Since $\rho_{M}(\gamma)^{\ast}=D_{\gamma^{\ast}%
}\mid_{\Lambda_{3}(M^{\ast})}$ for $\gamma\in gl(M)$, we see that $\rho
_{M}(gl(M))$ also acts as derivations of $\delta(x,\xi)(y)$. \ Clearly,
$Id_{\Lambda_{3}(M)}$ acts as derivations of the triple product, so
$[\alpha,\delta(x,\xi)]=\delta(\alpha x,\xi)+\delta(x,-\alpha^{\ast}\xi)$,
showing case 1.

Case 2: $J(\alpha,x,y)$. \ As above, $\widetilde{sl}(M)$ acts as derivations
of $\mu\cdot u=1$ and $(xy)\cdot\mu$. \ Thus,
\[
0=(d_{-\alpha^{\ast}}\mu)\cdot u+\mu\cdot(d_{\alpha}u)=(d_{-\alpha^{\ast}}%
\mu)\cdot u,
\]
so $d_{-\alpha^{\ast}}\mu=0$, and%
\[
\alpha((xy)\cdot\mu)=((\alpha x)y)\cdot\mu+(x(\alpha y))\cdot\mu
+(xy)\cdot(d_{-\alpha^{\ast}}\mu),
\]
so $[\alpha\lbrack x,y]]=[\alpha x,y]+[x,\alpha y]\,$.

Case 3 : $J(x,y,\xi)$. \ We see by Lemma \ref{lem: e properties}(vi) that%
\begin{align*}
\lbrack\lbrack x,y],\xi]  & =(((xy)\cdot\mu)\xi)\cdot u=-(\xi((xy)\cdot
\mu))\cdot u\\
& =-\xi\cdot(((xy)\cdot\mu)\cdot u)=-\xi\cdot(xy)\\
& =-(\xi\cdot x)y+D_{e(x,\xi)}y-D_{e(y,\xi)}x+(\xi\cdot y)x\\
& =\delta(x,\xi)(y)-\delta(y,\xi)(x)\\
& =[[x,\xi],y]-[[y,\xi],x].
\end{align*}

Case 4: $J(x,y,z)$. \ We have%

\begin{align*}
\lbrack\lbrack x,y],z]  & =-\delta(z,(xy)\cdot\mu)=-\rho_{M}(e(z,(xy)\cdot
\mu))+z\cdot((xy)\cdot\mu)Id_{\Lambda_{3}(M)}\\
& =-\rho_{M}(e(z,(xy)\cdot\mu))+((xyz)\cdot\mu)Id_{\Lambda_{3}(M)}.
\end{align*}
Also, by Lemma \ref{lem: e properties}(v) and Lemma \ref{lem: rank n}(i),
\[
\sum\limits_{x,y,z\circlearrowleft}e(x,(yz)\cdot\mu)=e(xyz,\mu)=((xyz)\cdot
\mu)Id_{M}.
\]
Thus,%
\[
\sum\limits_{x,y,z\circlearrowleft}[[x,y],z]=-((xyz)\cdot\mu)\rho_{M}%
(Id_{M})+3((xyz)\cdot\mu)Id_{\Lambda_{3}(M)}=0.
\]

\end{proof}

\bigskip

Suppose $\omega:M\rightarrow N$ is a $\sigma$-semilinear homomorphism where
$\sigma$ is an automorphism of $k$. \ Extending the definition for linear
maps, we define the $\sigma^{-1}$-semilinear map $\omega^{\ast}:N^{\ast
}\rightarrow M^{\ast}$ with $\omega^{\ast}(\phi)=\sigma^{-1}\phi\omega$. \ Let
$\theta_{\omega}$ be the unique extension of $\omega$ to a $\sigma$-semilinear
homomorphism $\Lambda(M)\rightarrow\Lambda(N)$. \ Note $\theta_{\omega
}(a)=\sigma(a)$ for $a\in k$.

\begin{lemma}
Let $M,u$ be as in Theorem \ref{thm: split}. \ The map
\begin{equation}
\alpha\oplus x\oplus\xi\rightarrow-\alpha^{\ast}\oplus\xi\oplus
x\label{eq: Lie isom M M*}%
\end{equation}
for $\alpha\in\widetilde{sl}(M)$, $x\in\Lambda_{3}(M)$, $\xi\in\Lambda
_{3}(M^{\ast})$ is an isomorphism $\mathcal{G}(M,u)\rightarrow\mathcal{G}%
(M^{\ast},\mu)$. \ If $\omega:M\rightarrow N$ is a $\sigma$-semilinear
isomorphism, then
\begin{equation}
\alpha\oplus x\oplus\xi\rightarrow\theta_{\omega}\alpha\theta_{\omega}%
^{-1}\oplus\theta_{\omega}x\oplus\theta_{\omega^{\ast-1}}\xi
\label{eq: semilinear Lie isom}%
\end{equation}
for $\alpha\in\widetilde{sl}(M)$, $x\in\Lambda_{3}(M)$, $\xi\in\Lambda
_{3}(M^{\ast})$ is a $\sigma$-semilinear isomorphism $\mathcal{G}%
(M,u)\rightarrow\mathcal{G}(N,\theta_{\omega}u,)$.
\end{lemma}

\begin{proof}
Using (\ref{eq: rho star}) and Lemma \ref{lem: e properties}, we see
$\delta(x,\xi)^{\ast}=\delta(\xi,x)$. \ It is then clear that
(\ref{eq: Lie isom M M*}) is an isomorphism. \ 

The Lie product on $\mathcal{G}(M,u)$ is completely determined by the graded
products on $\Lambda(M)$ and $\Lambda(M^{\ast})$, the actions of
$\Lambda(M^{\ast})$ on $\Lambda(M)$ and $\Lambda(M)$ on $\Lambda(M^{\ast})$,
the actions $\beta\rightarrow\rho_{M}(\beta)=D_{\beta}\mid_{\Lambda_{3}(M)}$
and $\beta\rightarrow$ $-\rho_{M}(\beta)^{\ast}$of $gl(M)$ on $\Lambda_{3}(M)$
and $\Lambda_{3}(M^{\ast})$, and the elements $u\in\Lambda_{9}(M)$, $\mu
\in\Lambda(M^{\ast})$. \ Thus, if $\eta:\Lambda(M)\rightarrow\Lambda(N)$ and
$\eta^{\prime}:\Lambda(M^{\ast})\rightarrow\Lambda(N^{\ast})$ are graded ring
isomorphisms and $\breve{\eta}:gl(M)\rightarrow gl(N)$ is a Lie ring
isomorphism with
\begin{align}
\eta(\xi\cdot x)  & =\eta^{\prime}(\xi)\cdot\eta(x),\label{eq: isom 1}\\
\eta^{\prime}(x\cdot\xi)  & =\eta(x)\cdot\eta^{\prime}(\xi),\label{eq: isom 2}%
\\
\rho_{N}(\breve{\eta}(\beta))  & =\eta\rho_{M}(\beta)\eta^{-1}%
,\label{eq: isom 3}\\
\rho_{N}(\breve{\eta}(\beta))^{\ast}  & =\eta^{\prime}\rho_{M}(\beta)^{\ast
}\eta^{\prime-1},\label{eq: isom 4}%
\end{align}
for $x\in\Lambda_{3}(M)$, $\xi\in\Lambda_{3}(M^{\ast})$, and $\beta\in gl(M)$,
then
\[
\alpha\oplus x\oplus\xi\rightarrow\eta\alpha\eta^{-1}\oplus\eta x\oplus
\eta^{\prime}\xi
\]
is a Lie ring isomorphism $\mathcal{G}(M,u)\rightarrow\mathcal{G}(N,\eta u)$.
\ Now let $\eta=\theta_{\omega}$, $\eta^{\prime}=\theta_{\omega^{\ast-1}}$,
and $\breve{\eta}(\beta)=\omega\beta\omega^{-1}$. \ We can rewrite
(\ref{eq: isom 1}) as
\begin{equation}
\theta_{\omega}\Delta_{\xi}\theta_{\omega}^{-1}=\Delta_{\theta_{\omega
^{\ast-1}}(\xi)}.\label{eq: isom 1'}%
\end{equation}
Since both sides of (\ref{eq: isom 1'}) are multiplicative in $\xi$, we can
assume $\xi\in M^{\ast}$. \ In that case, both sides are antiderivations of
$\Lambda(N)$, so it suffices to apply both sides to $\theta_{\omega}(M)=N$.
\ We have
\begin{align*}
\Delta_{\theta_{\omega^{\ast-1}}(\xi)}\theta_{\omega}(m)  & =\omega^{\ast
-1}(\xi)(\omega(m))=(\sigma\xi\omega^{-1})(\omega(m))\\
& =\sigma\xi(m)=\theta_{\omega}(\xi(m))=\theta_{\omega}\Delta_{\xi}(m),
\end{align*}
and (\ref{eq: isom 1}) follows. \ Reversing the roles of $M$ and $M^{\ast}$
gives (\ref{eq: isom 2}). \ If $\beta\in gl(M)$, then $\theta_{\omega}%
D_{\beta}\theta_{\omega}^{-1}=D_{\omega\beta\omega^{-1}}$ since they are
derivations agreeing on $\theta_{\omega}(M)=N$. \ \ This shows
(\ref{eq: isom 3}). \ Finally,
\begin{align*}
\rho_{N}(\omega\beta\omega^{-1})^{\ast}  & =\rho_{N^{\ast}}((\omega\beta
\omega^{-1})^{\ast})=\rho_{N^{\ast}}(\omega^{\ast-1}\beta^{\ast}\omega^{\ast
})\\
& =\theta_{\omega^{\ast-1}}\rho_{M}(\beta)^{\ast}\theta_{\omega^{\ast}},
\end{align*}
showing (\ref{eq: isom 4}). \ Thus, the $\sigma$-semilinear map
(\ref{eq: semilinear Lie isom}) is a Lie isomorphism.
\end{proof}

\bigskip

Let $K$ be a unital commutative ring with involution $a\rightarrow\bar{a}$ and
let $k$ be the subring of fixed elements. \ Let $M$ be a finite generated
projective $K$-module.of rank $9$ with a nonsingular hermitian form $h$; i.e.,
$\eta:m\rightarrow h(m,)$ is a semilinear isomorphism $M\rightarrow M^{\ast}$.
\ Define the semilinear involution $\tau$ on $gl(M) $ by $h(m,\alpha
n)=h(\tau(\alpha)m,n)$; i.e., $\tau(\alpha)=\eta^{-1}\alpha^{\ast}\eta$.
\ Let
\begin{align*}
u(M,h)  & =\{\alpha\in gl(M):\tau(\alpha)=-\alpha\},\\
su(M,h)  & =u(M,h)\cap sl(M),\\
sk(K)  & =\{a\in K:\bar{a}=-a\},\\
\tilde{u}(M,h)  & =\rho_{M}(u(M,h))+sk(K)Id_{\Lambda_{3}(M)}.
\end{align*}
Clearly, $\tilde{u}(M,h)$ is a subalgebra of $\widetilde{gl}(M)$. \ Note,
$sk(K)Id_{M}\subset u(M,h)$, so $\tilde{u}(M,h)=\rho_{M}(u(M,h))$ if $\frac
{1}{3}\in K$. \ Finally, set
\[
\widetilde{su}(M,h)=\tilde{u}(M,h)\cap\widetilde{sl}(M).\
\]
We also set $x\cdot y=\theta_{\eta}(x)\cdot y$ for $x,y\in\Lambda(M)$ and
$\delta(x,y)=\delta(x,\theta_{\eta}(y))$ for $x,y\in\Lambda_{3}(M)$.

\begin{theorem}
\label{thm: twisted Lie algebra}Let $K$ be a unital commutative ring with
involution $a\rightarrow\bar{a}$ and let $k$ be the subring of fixed elements.
\ Let $M$ be a finite generated projective $K$-module.of rank $9$ with a
nonsingular hermitian form $h$. \ If $u\in\Lambda_{9}(M)$ with $u\cdot u=1$
and $\mu=\theta_{\eta}(u)$, then%
\[
\zeta(\alpha\oplus x\oplus\xi)=-\theta_{\eta}^{-1}\alpha^{\ast}\theta_{\eta
}\oplus\theta_{\eta}^{-1}(\xi)\oplus\theta_{\eta}(x)
\]
for $\alpha\in\widetilde{sl}(M)$, $x\in\Lambda_{3}(M)$, $\xi\in\Lambda
_{3}(M^{\ast})$ is a semi-linear automorphism of $\mathcal{G}(M,u)$.
\ Moreover, $\alpha\oplus x\oplus\theta_{\eta}(x)\rightarrow\alpha\oplus x$ is
an isomorphism of the Lie algebra $\mathcal{G}(\zeta)$ over $k$ of fixed
points of $\zeta$ to
\[
\mathcal{G}(M,h,u)=\widetilde{su}(M,h)\oplus\Lambda_{3}(M)
\]
with skew-symmetric product given by%
\begin{align*}
\lbrack\alpha,\beta]  & =\alpha\beta-\beta\alpha,\\
\lbrack\alpha,x]  & =\alpha x,\\
\lbrack x,y]  & =(\delta(x,y)-\delta(y,x))\oplus(xy)\cdot u
\end{align*}
for $\alpha,\beta\in\widetilde{su}(M,h)$, $x,y\in\Lambda_{3}(M)$.
\end{theorem}

\begin{proof}
Since $h$ is hermitian, it is easy to see that $\eta^{\ast}=\eta$ and
$(\theta_{\eta}\alpha\theta_{\eta}^{-1})^{\ast}=\theta_{\eta}^{-1}\alpha
^{\ast}\theta_{\eta}$. \ Thus, $\zeta$ is the product of the semilinear
isomorphism $\mathcal{G}(M,u)\rightarrow\mathcal{G}(M^{\ast},\mu)$ given by
(\ref{eq: semilinear Lie isom}) with $N=M^{\ast}$ and $\omega=\eta$ and the
inverse of the isomorphism (\ref{eq: Lie isom M M*}). Since
\begin{align*}
\theta_{\eta}^{-1}\rho_{M}(\alpha)^{\ast}\theta_{\eta}  & =\rho_{M}(\eta
^{-1}\alpha^{\ast}\eta)=\rho_{M}(\tau(\alpha)),\\
\theta_{\eta}^{-1}(aId_{\Lambda_{3}(M)}^{\ast})\theta_{\eta}  & =\bar
{a}Id_{\Lambda_{3}(M)},
\end{align*}
we see that the Lie algebra $\mathcal{G}(\zeta)$ of fixed points of $\zeta$ is%
\[
\mathcal{G}(\zeta)=\{\alpha\oplus x\oplus\theta_{\eta}(x):\alpha\in
\widetilde{su}(M,h),x\in\Lambda_{3}(M)\}.
\]
The $\widetilde{su}(M,h)$ component of $[x\oplus\theta_{\eta}(x),y\oplus
\theta_{\eta}(y)]$ is%
\[
\lbrack x,\theta_{\eta}(y)]-[y,\theta_{\eta}(x)]=\delta(x,y)-\delta(y,x),
\]
while the $\Lambda_{3}(M)$ component is%
\begin{align*}
\lbrack\theta_{\eta}(x),\theta_{\eta}(y)]  & =(\theta_{\eta}(x)\theta_{\eta
}(y))\cdot u=\theta_{\eta}(xy)\cdot u\\
& =(xy)\cdot u.
\end{align*}
Thus, $\alpha\oplus x\oplus\theta_{\eta}(x)\rightarrow\alpha\oplus x$ is an
isomorphism of $\mathcal{G}(\zeta)$ with $\mathcal{G}(M,h,u)$.
\end{proof}

\bigskip

\section{Forms of exceptional Lie algebras}

\begin{lemma}
\label{lem: isom of extensions}If $F\in k$-alg is faithfully flat, then there
are canonical isomorphisms%
\begin{align*}
\mathcal{G}(M,u)_{F}  & \cong\mathcal{G}(M_{F},u_{F}),\\
\mathcal{G}(M,h,u)_{F}  & \cong\mathcal{G}(M_{F},h_{F},u_{F}),
\end{align*}
where $u_{F}$ is the image of $u\otimes1$ in the canonical isomorphism
$\Lambda_{9}(M)_{F}\rightarrow\Lambda_{9}(M_{F})$ and $h_{F}$ is the extension
of the $k$-bilinear map $h$ given by (\ref{eq: multilinear extension}).
\end{lemma}

\begin{proof}
Since $M$ is finitely generated projective, we have seen that there are
canonical isomorphisms%
\begin{align}
\Lambda_{3}(M)_{K}  & \cong\Lambda_{3}(M_{K}),\label{eq: can isom 1}\\
\Lambda_{3}(M^{\ast})_{K}  & \cong\Lambda_{3}(M_{K}^{\ast}%
),\label{eq: can isom 2}\\
gl(M)_{K}  & \cong gl(M_{K}),\label{eq: can isom 3}%
\end{align}
for $K\in k$-alg. \ Moreover, $(\rho(gl(M)))_{K}\cong\rho_{K}(gl(M_{K}))$ for
$\rho:\alpha\rightarrow D_{\alpha}\mid_{\Lambda_{3}(M)}$, so
\[
\widetilde{gl}(M)_{K}\cong\widetilde{gl}(M_{K}).
\]
\ If $F\in k$-alg is faithfully flat, then the exact sequence%
\[
\widetilde{sl}(M)\rightarrow\widetilde{gl}(M)\overset{T}{\rightarrow}k
\]
implies that%
\[
\widetilde{sl}(M)_{F}\rightarrow\widetilde{gl}(M)_{F}\overset{T_{F}%
}{\rightarrow}F
\]
is exact. \ Thus, $\widetilde{sl}(M)_{F}=\ker(T_{F})\cong\widetilde{sl}%
(M_{F})$ . \ Similarly, $\widetilde{su}(M,h)$ is the kernel of the map
$\alpha\rightarrow(\alpha+\tau(\alpha))\oplus T(\alpha)$, so $\widetilde
{su}(M,h)_{F}\cong\widetilde{su}(M_{F},h_{F})$. \ The canonical isomorphisms
of the lemma are now obvious.
\end{proof}

\bigskip

Suppose $K=k_{+}\oplus k_{-}$ where $k_{\sigma}$ is an isomorphic copy of $k $
via $a\rightarrow a_{\sigma}$ and $\bar{a}_{\sigma}=a_{-\sigma}$ for
$\sigma=\pm$. \ We shall identify $a\in k$ with $a_{+}\oplus a_{-}\in K$, and
write $M_{\sigma}=1_{\sigma}M$ and $m_{\sigma}=1_{\sigma}m$ where $M$ is a
$K$-module and $m\in M$. \ 

\begin{lemma}
\label{lem: split K}If $M,h,u$ and $\zeta$ are as in Theorem
\ref{thm: twisted Lie algebra} for $K=k_{+}\oplus k_{-}$, then%
\[
\alpha\oplus x\rightarrow\alpha_{+}\oplus x_{+}\oplus\theta_{\eta}(x_{-})
\]
is an isomorphism of $\mathcal{G}(M,h,u)$ with $\mathcal{G}(M_{+},u_{+})$.
\end{lemma}

\begin{proof}
Clearly, $\mathcal{G}(M,u)=\mathcal{G}(M,u)_{+}\oplus\mathcal{G}(M,u)_{-}$ as
Lie algebras over $K$. \ Moreover, since $\zeta$ is semilinear, $\zeta$
interchanges $\mathcal{G}(M,u)_{+}$ with $\mathcal{G}(M,u)_{-}$, so
\[
\mathcal{G}(\zeta)=\{z+\zeta(z):z\in\mathcal{G}(M,u)_{+}\}.
\]
Thus, $z\rightarrow z_{+}$ is a Lie algebra isomorphism $\mathcal{G}%
(\zeta)\rightarrow\mathcal{G}(M,u)_{+}$ over $k$. \ Using Theorem
\ref{thm: twisted Lie algebra}, we see that
\[
\alpha\oplus x\rightarrow(\alpha\oplus x\oplus\theta_{\eta}(x))_{+}=\alpha
_{+}\oplus x_{+}\oplus\theta_{\eta}(x_{-})
\]
is a Lie algebra isomorphism $\mathcal{G}(M,h,u)\rightarrow\mathcal{G}%
(M,u)_{+}$ over $k$. \ On the other hand, $M_{k_{+}}=1_{+}\otimes M_{+}$ can
be identified with $M_{+}$ as $k_{+}$-modules. \ Thus,
\[
\mathcal{G}(M,u)_{+}=\mathcal{G}(M,u)_{k_{+}}=\mathcal{G}(M_{k_{+}}%
,1_{+}\otimes u)_{+}=\mathcal{G}(M_{+},u_{+})
\]
as Lie algebra over $k_{+}$ and hence over $k$.
\end{proof}

\bigskip

Suppose $M,u,\mu$ are as in Theorem \ref{thm: split} and $M$ is free over $k
$. \ Let $B=\{m_{1},\ldots,m_{9}\}$ be a basis for $M$ and $\phi_{1}%
,\ldots,\phi_{9}$ the dual basis of $M^{\ast}$; i.e., $\phi_{i}(m_{j}%
)=\delta_{ij}$. \ For
\[
S=\{i_{1}<\cdots<i_{l}\}\subset\{1,\ldots,9\},
\]
let
\begin{align*}
m_{S}  & =m_{i_{1}}\cdots m_{i_{l}},\\
\phi_{S}  & =\phi_{i_{l}}\cdots\phi_{i_{1}},
\end{align*}
so $\{m_{S}:\left\vert S\right\vert =l\}$ and $\{\phi_{S}:\left\vert
S\right\vert =l\}$ are dual bases for $\Lambda_{l}(M)$ and $\Lambda
_{l}(M^{\ast})$. \ Set $u_{B}=m_{\{1,\ldots,9\}}$ and $\mu_{B}=\phi
_{\{1,\ldots,9\}}$. Since $u=au_{B}$, $\mu=b\mu_{B}$ and $1=\mu\cdot u=ab $,
so $a$ and $b$ are invertible, we may replace $m_{1}$ by $am_{1}$ and
$\phi_{1}$ by $b\phi_{1}$ to assume that $u=u_{B}$ and $\mu=\mu_{B}$. \ Now
$e_{ij}:=e(m_{i},\phi_{j})$, $1\leq i,j\leq9$ is a basis for $gl(M) $ and the
matrix of $e_{ij}$ relative to the basis for $M$ is just the usual matrix
unit. \ Let%
\begin{align*}
h_{1}  & =\rho(e_{11}+e_{22}+e_{33})-Id_{\Lambda_{3}(M)},\\
h_{i}  & =\rho(e_{ii}-e_{i-1,i-1})\text{ for }2\leq i\leq8.
\end{align*}

\begin{lemma}
\label{lem: bases}If $M$ is a free module with basis $B=\{m_{1},\ldots
,m_{9}\}$, then
\[
\tilde{B}=\{h_{i}:1\leq i\leq8\}\cup\{\rho(e_{ij}):i\neq j\}
\]
is a basis for $\widetilde{sl}(M)$ and
\[
\hat{B}=\tilde{B}\cup\{m_{S}:\left\vert S\right\vert =3\}\cup\{\phi
_{S}:\left\vert S\right\vert =3\}
\]
is a basis for $\mathcal{G}(M,u_{B})$. \ Thus, $\mathcal{G}(M,u_{B})_{K}$ is
canonically isomorphic to \linebreak$\mathcal{G}(M_{K},u_{B\otimes1}$.$)$ for
any $K\in k$-alg.
\end{lemma}

\begin{proof}
First, note $T(h_{1})=3-3=0$, so $h_{1}\in\widetilde{sl}(M)$. \ Suppose
$\alpha=\sum_{i,j}a_{ij}e_{ij}\in gl(M)$ and $b\in k$ with $\rho
(\alpha)+bId_{\Lambda_{3}(M)}=0$. \ If $i\neq j$, choose $k,s$ with $i,j,k,s$
distinct. \ We see that $\beta=\rho(e_{ij})$ is the only element among
$\rho(e_{pq}),Id_{\Lambda_{3}(M)}$ with $\beta(m_{j}m_{k}m_{s})$ having a
nonzero coefficient of $m_{i}m_{k}m_{s}$. \ Thus, $a_{ij}=0$ for $i\neq j$.
\ Also,
\[
\rho(\alpha)m_{i}m_{j}m_{k}=\sum_{p=1}^{9}a_{pp}\rho(e_{pp})m_{i}m_{j}%
m_{k}=(a_{ii}+a_{jj}+a_{kk})m_{i}m_{j}m_{k},
\]
so $a_{ii}+a_{jj}+a_{kk}=-b$ for distinct $i,j,k$. \ Thus, $a_{ii}=a$ and
$b=-3a$ for $a=a_{11}$. \ Now suppose
\[
\sum_{i=1}^{8}c_{i}h_{i}+\sum_{1\leq i\neq j\leq9}c_{ij}\rho(e_{ij})=0.
\]
Letting
\begin{align*}
\alpha & =c_{1}(e_{11}+e_{22}+e_{33})+\sum_{i=2}^{8}c_{i}(e_{ii}%
-e_{i-1,i-1})+\sum_{1\leq i\neq j\leq9}c_{ij}e_{ij}\\
& =\sum_{i,j}a_{ij}e_{ij},
\end{align*}
we have $\rho(\alpha)-c_{1}Id_{\Lambda_{3}(M)}=0$. Thus, $c_{ij}=a_{ij}=0$ for
$i\neq j$. \ Also, $a_{99}=0$, so all $a_{ii}=0$ and $c_{1}=-3a_{11}=0$.
\ Moreover, $\sum_{i=2}^{8}c_{i}(e_{ii}-e_{i-1,i-1})=0$ forces all $c_{i}=0$.
\ Thus, $\tilde{B}$ is independent. \ To show that it spans $\widetilde
{sl}(M)$, suppose $\alpha=\sum_{i,j}a_{ij}e_{ij}$ and $x=\rho(\alpha
)+bId_{\Lambda_{3}(M)}\in\widetilde{sl}(M)$; i.e., $tr(\alpha)+3b=0$. \ After
subtracting $a_{99}(\rho(Id_{M})-3Id_{\Lambda_{3}(M)})=0$, we may assume
$a_{99}=0$. \ Subtracting $a_{ij}\rho(e_{ij})$ for $i\neq j$ and $-bh_{1}$, we
can also assume $a_{ij}=0$ for $i\neq j$ and $b=0$. \ Thus, $tr(\alpha)=0$ and
$\rho(\alpha)$ is in the span of $h_{2},\ldots,h_{8}$. \ Thus, $\tilde{B}$ is
a basis for $\widetilde{sl}(M)$, and hence $\hat{B}$ is a basis for
$\mathcal{G}(M,u_{B})$.

Now $B\otimes1:=\{m\otimes1:m\in B\}$ is a basis for $M_{K}$ and $\hat
{B}\otimes1$ is a basis for $\mathcal{G}(M,u_{B})_{K}$. \ The natural
bijection between $\hat{B}\otimes1$ and the basis $\widehat{B\otimes1}$ of
$\mathcal{G}(M_{K},u_{B\otimes1})$ induces a canonical isomorphism
$\mathcal{G}(M,u_{B})_{K}\rightarrow\mathcal{G}(M_{K},u_{B\otimes1})$.
\end{proof}

\bigskip

We remark that the rank of $\mathcal{G}(M,u)$ is $8+9\cdot8+\binom{9}%
{3}+\binom{9}{3}=80+2\cdot84=248$.

\bigskip

\begin{theorem}
\label{thm: type E8}Let $\mathbb{C}^{9}$ be the complex vector space of
dimension $9$ with standard basis $C$. \ Then $\mathcal{G}(\mathbb{C}%
^{9},u_{C})$ is a simple Lie algebra of type $E_{8}$ and $\hat{C}$ is a
Chevalley basis.
\end{theorem}

\begin{proof}
Let $M=\mathbb{C}^{9}$, $C=\{m_{1},\ldots,m_{9}\}$, $u=u_{C}$, and $\mu
=\mu_{C}$. \ Since $\frac{1}{3}\in\mathbb{C}$, $\rho:sl(M)\rightarrow
\widetilde{sl}(M)$ is an isomorphism. \ Now $\widetilde{sl}(M)$, $\Lambda
_{3}(M)$, and $\Lambda_{3}(M^{\ast})$ are nonisomorphic irreducible
$\widetilde{sl}(M)$-modules, so they are the only irreducible $\widetilde
{sl}(M)$-modules in $\mathcal{G}(M,u)$. \ Thus, if $I$ is a nonzero ideal of
$\mathcal{G}(M,u)$, then complete reducibility shows that $I$ contains at
least one of these submodules. \ Moreover,
\begin{align*}
0  & \neq[\widetilde{sl}(M),\Lambda_{3}(M)]\subset\Lambda_{3}(M),\\
0  & \neq[\widetilde{sl}(M),\Lambda_{3}(M^{\ast})]\subset\Lambda_{3}(M^{\ast
}),\\
0  & \neq[\Lambda_{3}(M),\Lambda_{3}(M^{\ast})]\subset\widetilde{sl}(M),
\end{align*}
so $I$ contains each of these submodules. \ Thus, $\mathcal{G}(M,u)$ is
simple. \ Let $\mathcal{H}$ be the trace $0$ diagonal maps of $M$ relative to
the given basis, so $\mathcal{H}$ is a Cartan subalgebra of $sl(M)$, and
$\mathcal{\tilde{H}=\rho(H)}$ is a Cartan subalgebra of $\widetilde{sl}(M)$.
\ Since $h_{1}=\rho(e_{11}+e_{22}+e_{33}-\frac{1}{3}Id_{M})$, we see $h_{i}$,
$1\leq i\leq8$ is a basis for $\mathcal{\tilde{H}}$. \ The centralizer of
$\mathcal{\tilde{H}}$ in $\mathcal{G}(M,u)$ is contained in $\widetilde
{sl}(M)$ and is hence $\mathcal{\tilde{H}}$. \ Thus, $\mathcal{\tilde{H}}$ is
a Cartan subalgebra of $\mathcal{G}(M,u)$. \ Let $\varepsilon_{i}%
\in\mathcal{\tilde{H}}^{\ast}$ with $\varepsilon_{i}(h)=a_{i}$ where
$\rho^{-1}(h)=diag(a_{1},\ldots,a_{9})\in\mathcal{H}$, as a diagonal matrix.
\ Clearly, $\sum_{i=1}^{9}\varepsilon_{i}=0$. \ We see that the roots $\Sigma$
of $\mathcal{\tilde{H}}$ for $\mathcal{G}(M,u)$ are all $\varepsilon
_{i}-\varepsilon_{j}$ for $i\neq j$ (in $\widetilde{sl}(M)$) and all
$\pm(\varepsilon_{i}+\varepsilon_{j}+\varepsilon_{k})$ for distinct $i,j,k$
(in $\Lambda_{3}(M)$ and $\Lambda_{3}(M^{\ast})$). \ Let $\alpha
_{1}=\varepsilon_{1}+\varepsilon_{2}+\varepsilon_{3}$ and $\alpha
_{i}=\varepsilon_{i}-\varepsilon_{i-1}$ for $2\leq i\leq8$. \ Now
$\Pi=\{\alpha_{1},\ldots,\alpha_{8}\}$ is a basis of $\mathcal{\tilde{H}%
}^{\ast}$. \ Moreover, an examination of the $\alpha_{j}$-string through
$\alpha_{i}$ shows that $\Pi$ is a fundamental system of roots with Dynkin
diagram $E_{8}$ with $\alpha_{2},\ldots,\alpha_{8}$ forming a diagram of type
$A_{7}$ and $\alpha_{1}$ connected to $\alpha_{4}$. \ Hence, $\mathcal{G}%
(M,u)$ is a Lie algebra of type $E_{8}$. \ To show that $\hat{C}$ is a
Chevalley basis, we need to show (\cite{H72}, p. 147)

\qquad(a) for each root $\alpha$, there is $x_{\alpha}\in\hat{C}%
\cap\mathcal{G}(M,u)_{\alpha}$,

\qquad(b) $[x_{\alpha},x_{-\alpha}]=h_{\alpha}$ with $[h_{\alpha},x_{\alpha
}]=2x_{\alpha}$,

\qquad(c) $h_{\alpha_{i}}=h_{i}$,

\qquad(d) the linear map with $x_{\alpha}\rightarrow-x_{-\alpha}$,
$h_{i}\rightarrow-h_{i}$ is an automorphism of $\mathcal{G}(M,u)$.

\noindent Clearly, $x_{\alpha}=\rho(e_{ij})$ for $\alpha=\varepsilon
_{i}-\varepsilon_{j}$, $x_{\alpha}=m_{S}$ and $x_{-\alpha}=\phi_{S}$ for
$\alpha=\varepsilon_{i}+\varepsilon_{j}+\varepsilon_{k}$ and $S=\{i<j<k\} $
satisfies (a). \ Now $[[e_{ij},e_{ji}],e_{ij}]=[e_{ii}-e_{jj},e_{ij}]=2e_{ij}%
$, so (b) holds for $\alpha=\varepsilon_{i}-\varepsilon_{j}$ and (c) holds for
$i\neq1$. \ Lemma \ref{lem: e properties}(v) with $l=1$ shows
\begin{align*}
e(m_{S},\phi_{S})  & =e(m_{i}m_{j}m_{k},\phi_{k}\phi_{j}\phi_{i}%
)=\sum\limits_{i,j,k\circlearrowleft}e(m_{i},(m_{j}m_{k})\cdot(\phi_{k}%
\phi_{j}\phi_{i}))\\
& =e_{ii}+e_{jj}+e_{kk}.
\end{align*}
Thus,
\begin{align*}
\lbrack m_{S},\phi_{S}]  & =\rho(e(m_{S},\phi_{S})-\frac{1}{3}(m_{S}\cdot
\phi_{S})Id_{M})\\
& =\rho(e_{ii}+e_{jj}+e_{kk}-\frac{1}{3}Id_{M}),
\end{align*}
so (b) holds for $\alpha=\pm(\varepsilon_{i}+\varepsilon_{j}+\varepsilon_{k})$
and (c) holds for $i=1$. \ Finally, let $\mathbb{C}$ have the trivial
involution and let $h$ be the symmetric bilinear form on $M$ with
$h(m_{i},m_{j})=\delta_{ij}$. \ Thus, $\eta$ as in Theorem
\ref{thm: twisted Lie algebra} has $\eta(m_{i})=\phi_{i}$. \ Now $\theta
_{\eta}(m_{C})=\phi_{1}\cdots\phi_{9}=\phi_{9}\cdots\phi_{1}=\phi_{C}$, and we
have an automorphism $\zeta$ given by Theorem \ref{thm: twisted Lie algebra}.
\ Since $\theta_{\eta}^{-1}\rho(\beta)^{\ast}\theta_{\eta}=\rho(\tau(\beta))$
for $\beta\in sl(M)$ where $\tau(e_{ij})=e_{ji}$, we see that $\zeta
(h_{i})=-h_{i}$, and $\zeta(x_{\alpha})=-x_{-\alpha}$ for $\alpha
=\varepsilon_{i}-\varepsilon_{j}$. \ Also, $\zeta(x_{\alpha})=\theta_{\eta
}(m_{S})=\phi_{i}\phi_{j}\phi_{k}=-\phi_{S}=-x_{-\alpha}$ for $\alpha
=\varepsilon_{i}+\varepsilon_{j}+\varepsilon_{k}$ and $S=\{i<j<k\}$. \ Thus,
(d) holds and $\hat{C}$ is a Chevalley basis.
\end{proof}

\bigskip

Let $\mathcal{G}(\mathbb{C)}$ be a simple Lie algebra over $\mathbb{C}$ of
type $X_{l}$ and let $\mathcal{G}(\mathbb{Z})$ be the $\mathbb{Z}$-span of a
Chevalley basis of $\mathcal{G}(\mathbb{C)}$. \ Up to isomorphism,
$\mathcal{G}(\mathbb{Z})$ is independent of the choice of Chevalley basis
(\cite{H72}, p. 150, Exercise 5). \ Set $\mathcal{G}(k)=\mathcal{G}%
(\mathbb{Z})_{k}$. \ We say that a Lie algebra $\mathcal{G}$ over $k$ is a
\textit{split form} of $X_{l}$ if $\mathcal{G\cong G}(k)$ and that
$\mathcal{G}$ is a \textit{form} of $X_{l}$ if $\mathcal{G}_{F}\cong
\mathcal{G}(F)$ for some faithfully flat $F\in k$-alg. \ If $F\in k$-alg and
$E\in F$-alg are faithfully flat, then $E\in k$-alg is faithfully flat.
\ Thus, if $\mathcal{G}_{F}$ is a form of $X_{l}$ for some faithfully flat
$F\in k$-alg, then $\mathcal{G}$ is a form of $X_{l}$.

\begin{corollary}
The Lie algebra $\mathcal{G}(M,u)$ in Theorem \ref{thm: split} is a form of
$E_{8}$ and is a split form if $M$ is free. \ If $K$ is a quadratic \'{e}tale
$k$-algebra, then the Lie algebra $\mathcal{G}(M,h,u)$ in Theorem
\ref{thm: twisted Lie algebra} is a form of $E_{8}$.
\end{corollary}

\begin{proof}
If $\hat{C}$ is the Chevalley basis of $\mathcal{G}(\mathbb{C}^{9},u_{C})$
given by Theorem \ref{thm: type E8}, we can identify $C$ with the standard
basis of $\mathbb{Z}^{9}$ and $\hat{C}$ with the corresponding basis for
$\mathcal{G}(\mathbb{Z}^{9},u_{C})$. \ In particular, $\mathcal{G}%
(\mathbb{Z}^{9},u_{C})=\mathcal{G}(\mathbb{Z})$, the $\mathbb{Z}$-span
$\hat{C}$. \ If $M,u$ are as in Theorem \ref{thm: split} with $M$ free, we can
choose a \ basis $B$ for $M$ with $u=u_{B}$ and $\mu=\mu_{B}$. \ The
isomorphism $M\rightarrow\mathbb{Z}_{k}^{9}\cong k^{9}$ taking $B$ to
$C\otimes1$ induces an isomorphism $\mathcal{G}(M,u_{B})\rightarrow
\mathcal{G}(\mathbb{Z}_{k}^{9},u_{C\otimes1})$. \ Since
\[
\mathcal{G}(k)=\mathcal{G}(\mathbb{Z})_{k}=\mathcal{G}(\mathbb{Z}^{9}%
,u_{C})_{k}\cong\mathcal{G}(\mathbb{Z}_{k}^{9},u_{C\otimes1}),
\]
by Lemma \ref{lem: bases}, we see that $\mathcal{G}(M,u)$ is a split form if
$M$ is free. \ For the general case, we know there is a faithfully flat $F\in
k$-alg with $M_{F}$ a free $k_{F}$-module of rank $9$ (\cite{B89}, II.5,
Exercise 8). \ By Lemma \ref{lem: isom of extensions} and the result for free
$M$, we see
\[
\mathcal{G}(M,u)_{F}\cong\mathcal{G}(M_{F},u_{F})\cong\mathcal{G}(F)
\]
and $\mathcal{G}(M,u)$ is a form of $E_{8}$.

For $M,h,u$ as in Theorem \ref{thm: twisted Lie algebra} with $K$ a quadratic
\'{e}tale $k$-algebra, we know\ by Proposition \ref{prop: quadratic etale}
that $K$ is faithfully flat and $K_{K}\cong K\oplus K$. \ Thus,
\begin{equation}
\mathcal{G}(M,h,u)_{K}\cong\mathcal{G}(M_{K},h_{K},u_{K})\cong\mathcal{G}%
((M_{K})_{+},(u_{K})_{+})\label{eq: quad etale untwists}%
\end{equation}
by Lemmas \ref{lem: isom of extensions} and \ref{lem: split K}, so
$\mathcal{G}(M,h,u)_{K}$ and hence $\mathcal{G}(M,h,u)$ are forms of $E_{8}$.
\end{proof}

\bigskip

\begin{theorem}
Let $M,u,\mu$ be as in Theorem \ref{thm: split}

$\qquad$(i) If $M=M_{1}\oplus M_{2}$ with $M_{1}$ of rank $3$ and $M_{2}$ of
rank $6$, then
\[
\mathcal{G}(M_{1},M_{2},u)=[M_{1}\Lambda_{2}(M_{2}),M_{1}^{\ast}\Lambda
_{2}(M_{2}^{\ast})]\oplus M_{1}\Lambda_{2}(M_{2})\oplus M_{1}^{\ast}%
\Lambda_{2}(M_{2}^{\ast})
\]
is a Lie subalgebra of $\mathcal{G}(M,u)$ and a form of $E_{7}$.

\qquad(ii) $M=M_{1}\oplus M_{2}\oplus M_{3}$ with each $M_{i}$ of rank $3$,
then%
\[
\mathcal{G}(M_{1},M_{2},M_{3},u)=[M_{1}M_{2}M_{3},M_{1}^{\ast}M_{2}^{\ast
}M_{3}^{\ast}]\oplus M_{1}M_{2}M_{3}\oplus M_{1}^{\ast}M_{2}^{\ast}M_{3}%
^{\ast}%
\]
is a Lie subalgebra of $\mathcal{G}(M,u)$ and a form of $E_{6}$.

Let $M,h,u$ as in Theorem \ref{thm: twisted Lie algebra} with $K$ a quadratic
\'{e}tale $k$-algebra. \ Set $d(x,y)=\delta(x,y)-\delta(y,x)$ for
$x,y\in\Lambda_{3}(M)$.

$\qquad$(iii) If $M=M_{1}\perp M_{2}$ with $M_{1}$ of rank $3$ and $M_{2}$ of
rank $6$, then
\[
\mathcal{G}(M_{1},M_{2},h,u)=d(M_{1}\Lambda_{2}(M_{2}),M_{1}\Lambda_{2}%
(M_{2}))\oplus M_{1}\Lambda_{2}(M_{2})
\]
is a Lie subalgebra of $\mathcal{G}(M,h,u)$ and a form of $E_{7}$.

\qquad(iv) $M=M_{1}\perp M_{2}\perp M_{3}$ with each $M_{i}$ of rank $3$, then%
\[
\mathcal{G}(M_{1},M_{2},M_{3},h,u)=d(M_{1}M_{2}M_{3},M_{1}M_{2}M_{3})\oplus
M_{1}M_{2}M_{3}%
\]
is a Lie subalgebra of $\mathcal{G}(M,h,u)$ and a form of $E_{6}$.
\end{theorem}

\begin{proof}
We show that $\mathcal{G}(M_{1},M_{2},M_{3},u)$ is a subalgebra, and the other
cases can be handled similarly. \ Since $M_{i}\cdot M_{j}^{\ast}=0$ for $i\neq
j$, we see%
\begin{align*}
& ((M_{1}M_{2}M_{3})(M_{1}M_{2}M_{3}))\cdot\Lambda_{9}(M)\\
& =((M_{1}M_{2}M_{3})(M_{1}M_{2}M_{3}))\cdot\Lambda_{3}(M_{1}^{\ast}%
)\Lambda_{3}(M_{2}^{\ast})\Lambda_{3}(M_{3}^{\ast})\\
& \subset M_{1}^{\ast}M_{2}^{\ast}M_{3}^{\ast}.
\end{align*}
Thus,%
\[
\lbrack M_{1}M_{2}M_{3},M_{1}M_{2}M_{3}]\subset M_{1}^{\ast}M_{2}^{\ast}%
M_{3}^{\ast}%
\]
and similarly%
\[
\lbrack M_{1}^{\ast}M_{2}^{\ast}M_{3}^{\ast},M_{1}^{\ast}M_{2}^{\ast}%
M_{3}^{\ast}]\subset M_{1}M_{2}M_{3}.\
\]
Also,%
\[
(M_{i}M_{j})\cdot(M_{1}^{\ast}M_{2}^{\ast}M_{3}^{\ast})\subset M_{k}^{\ast}%
\]
for $\{i,j,k\}=\{1,2,3\}$. \ Thus,
\[
e(M_{1}M_{2}M_{3},M_{1}^{\ast}M_{2}^{\ast}M_{3}^{\ast})\subset\sum_{i=1}%
^{3}e(M_{i},M_{i}^{\ast})
\]
by Lemma \ref{lem: e properties}(v). \ Since $\rho(e(M_{i},M_{i}^{\ast}))$
stabilizes $M_{1}M_{2}M_{3}$ and $\rho(e(M_{i},M_{i}^{\ast}))^{\ast}$
stabilizes $M_{1}^{\ast}M_{2}^{\ast}M_{3}^{\ast}$, we see $\mathcal{G}%
(M_{1},M_{2},M_{3},u)$ is a subalgebra.

Since $\mathcal{G}(M_{1},M_{2},h,u)$ is the subalgebra generated by
$M_{1}\Lambda_{2}(M_{2})$ and

\noindent$\mathcal{G}(M_{1},M_{2},M_{3},h,u)$ is the subalgebra generated by
$M_{1}M_{2}M_{3}$, we can use the isomorphism (\ref{eq: quad etale untwists})
to reduce cases (iii) and (iv) to cases (i) and (ii). \ In cases (i) or (ii),
there is a faithfully flat $F\in k$-alg with each $M_{iF}$ free of rank $3$ or
$6$. \ We can choose a basis $B=\{m_{1},\ldots,m_{9}\}$ for $M_{F}$ with
$1\otimes u=u_{B}$ and $1\otimes\mu=\mu_{B}$ which is compatible with the
direct sum decomposition; i.e., $M_{1F}=span_{F}(m_{1},m_{2},m_{3})$ and
$M_{2F}=span_{F}(m_{4},\ldots,m_{9}) $ or $M_{iF}=span_{F}(m_{3i-2}%
,m_{3i-1},m_{3i})$. \ The isomorphism $\mathcal{G}(M,u)_{F}\cong
\mathcal{G}(\mathbb{Z}^{9},u_{C})_{F}$ allows us to reduce to the cases
\begin{align*}
M  & =\mathbb{Z}^{9}=\mathbb{Z}^{(1,3)}\oplus\mathbb{Z}^{(4,9)},\\
M  & =\mathbb{Z}^{9}=\mathbb{Z}^{(1,3)}\oplus\mathbb{Z}^{(4,6)}\oplus
\mathbb{Z}^{(7,9)}%
\end{align*}
where $\mathbb{Z}^{(i,j)}=span_{\mathbb{Z}}(m_{i},\ldots,m_{j})$ for $1\leq
i\leq j\leq9$ and $C=\{m_{1},\ldots,m_{9}\}$ is the standard basis for
$\mathbb{Z}^{9}$.

Let $\mathcal{G}=\mathcal{G}(\mathbb{C}^{9},u_{C})$ as in Theorem
\ref{thm: type E8}. \ Let%
\begin{align*}
\beta_{i}  & =\alpha_{i}=\varepsilon_{i}-\varepsilon_{i-1}\text{ for
}i=2,3,5,6,7,\\
\beta_{1}  & =\alpha_{9}=\varepsilon_{9}-\varepsilon_{8},\\
\beta_{4}  & =\varepsilon_{2}+\varepsilon_{4}+\varepsilon_{8},\\
\beta_{8}  & =\varepsilon_{4}+\varepsilon_{5}+\varepsilon_{6}.\
\end{align*}
As before, by checking the $\beta_{j}$-string through $\beta_{i}$, we see that
$\tilde{\Pi}=\{\beta_{1},\ldots,\beta_{8}\}$ is a fundamental system of roots
with Dynkin diagram $E_{8}$ with $\beta_{2},\ldots,\beta_{8}$ forming a
diagram of type $A_{7}$ and $\beta_{1}$ connected to $\beta_{4}$. Moreover,
replacing $h_{i}$ in $\hat{C}$ by $\tilde{h}_{i}=h_{\beta_{i}}$, we get a
Chevalley basis $\tilde{C}$. \ Let
\begin{align*}
h^{\prime}  & =\rho(diag(-2,-2,-2,1,1,1,1,1,1)),\\
h^{\prime\prime}  & =\rho(diag(1,1,1,-1,-1,-1,0,0,0)).
\end{align*}
Since%
\begin{align*}
\beta_{i}(h^{\prime})  & =0\text{ for }1\leq i\leq7,\\
\beta_{8}(h^{\prime})  & =3,\\
\beta_{i}(h^{\prime\prime})  & =0\text{ for }1\leq i\leq6,\\
\beta_{7}(h^{\prime\prime})  & =1,
\end{align*}
we see that%
\[
\Sigma^{\prime}=\{\alpha\in\Sigma:\alpha(h^{\prime})=0\}
\]
is a root system of type $E_{7}$ and%
\[
\Sigma^{\prime\prime}=\{\alpha\in\Sigma:\alpha(h^{\prime})=\alpha
(h^{\prime\prime})=0\}
\]
is a root system of type $E_{6}$. \ Moreover, the subalgebra $\mathcal{G}%
^{\prime}$ generated by all $\mathcal{G}_{\alpha}$ with $\alpha\in
\Sigma^{\prime}$ is a complex simple Lie algebra of type $E_{7}$ with
Chevalley basis $\tilde{C}\cap\mathcal{G}^{\prime}$ and the subalgebra
$\mathcal{G}^{\prime\prime}$ generated by all $\mathcal{G}_{\alpha}$ with
$\alpha\in\Sigma^{\prime\prime}$ is a complex simple Lie algebra of type
$E_{6}$ with Chevalley basis $\tilde{C}\cap\mathcal{G}^{\prime\prime}$. \ We
see
\begin{align*}
\Sigma^{\prime}=\{\varepsilon_{i}-\varepsilon_{j}:1\leq i\neq j\leq3\text{ or
}4\leq i\neq j\leq9\}  & \\
\cup\ \ \{\pm(\varepsilon_{i}+\varepsilon_{j}+\varepsilon_{k}):1\leq
i\leq3\text{ and }4\leq j\neq k\leq9\},  & \\
\Sigma^{\prime\prime}=\{\varepsilon_{i}-\varepsilon_{j}:3l-2\leq i\neq
j\leq3l\text{ for }l=1,2,\text{ or }3\}  & \\
\cup\ \ \{\pm(\varepsilon_{i_{1}}+\varepsilon_{i_{2}}+\varepsilon_{i_{3}%
}):3l-2\leq i_{l}\leq3l\}.  &
\end{align*}
Since $[m_{i}m_{k}m_{l},\phi_{l}\phi_{k}\phi_{j}]=\rho(e_{ij})$ where
$C=\{m_{1},\ldots,m_{9}\}$, we see that the $\mathbb{Z}$-span of $\tilde
{C}\cap\mathcal{G}^{\prime}$ is generated as a $\mathbb{Z}$-algebra by
\[
\tilde{C}\cap(\mathbb{Z}^{(1,3)}\Lambda_{2}(\mathbb{Z}^{(4,9)})\cup
\mathbb{Z}^{(1,3)\ast}\Lambda_{2}(\mathbb{Z}^{(4,9)\ast}))
\]
while the $\mathbb{Z}$-span of $\tilde{C}\cap\mathcal{G}^{\prime\prime}$ is
generated as a $\mathbb{Z}$-algebra by
\[
\tilde{C}\cap(\mathbb{Z}^{(1,3)}\mathbb{Z}^{(4,6)}\mathbb{Z}^{(7,9)}%
\cup\mathbb{Z}^{(1,3)\ast}\mathbb{Z}^{(4,6)\ast}\mathbb{Z}^{(7,9)\ast}).
\]
In other words, $\mathcal{G}(\mathbb{Z}^{(1,3)},\mathbb{Z}^{(4,9)},u_{C})$ is
the $\mathbb{Z}$-span of $\tilde{C}\cap\mathcal{G}^{\prime}$ and%
\[
\mathcal{G}(\mathbb{Z}^{(1,3)},\mathbb{Z}^{(4,6)},\mathbb{Z}^{(7,9)},u_{C})
\]
is the $\mathbb{Z}$-span of $\tilde{C}\cap\mathcal{G}^{\prime\prime}$.
\end{proof}


\bigskip

\begin{thebibliography}{999}                                                                                              %
\bibitem[B88]{B88}N. Bourbaki, \textit{Elements of Mathematics, Algebra I,}
Chapters 1-3, Hermann, Paris, 1988.

\bibitem[B89]{B89}N. Bourbaki, \textit{Elements of Mathematics, Commutative
Algebra, }Chapters 1--7, Springer-Verlag, Berlin, 1989.

\bibitem[H72]{H72}James E Humphreys, \textit{Introduction to Lie Algebras and
Representation Theory,} Springer-Verlag, New York, 1972.
\end{thebibliography}
\end{document}